\documentclass[11pt, a4paper]{article}

\usepackage[T1]{fontenc}
\usepackage{amsmath}
\usepackage{amsthm}
\usepackage{amsfonts}
\usepackage{bbm}

\usepackage{cite}

\theoremstyle{plain}
\newtheorem{thm}{Theorem}

\newtheorem{prop}[thm]{Proposition}
\theoremstyle{definition}
\newtheorem{rem}[thm]{Remark}
\newtheorem{ex}[thm]{Example}

\def\P{\mathbb{P}}
\def\E{\mathbb{E}}
\def\R{\mathbb{R}}

\newcommand{\1}[1]{\mathbbm{1}_{#1}}
\newcommand{\di}{\,\mathrm{d}}

\let\BFseries\bfseries\def\bfseries{\BFseries\mathversion{bold}} 

\title{Ruin probabilities in the Cram\'er-Lundberg model with temporarily negative capital}

\author{Frank Aurzada \and Micha Buck}

\begin{document} 
\maketitle

\begin{abstract}
We study the asymptotics of the ruin probability in the Cram\'er-Lundberg
model with a modified notion of ruin. 
The modification is as follows. If the portfolio becomes negative, the asset is not immediately declared ruined but may survive due to certain mechanisms. Under a rather general assumption on the mechanism -- satisfied by most such modified models from
the literature -- we study the relation of the asymptotics of the modified
ruin probability to the classical ruin probability. This is done under the
Cram\'er condition as well as for subexponential integrated claim sizes. 
\end{abstract}

\section{Introduction}

The classical Cram\'er-Lundberg process $(U_t)_{t \geq 0}$ with 
\[
U_t = u + ct - \sum_{i=1}^{N_t} Y_i
\] 
is considered, where $u \geq 0$ denotes the initial capital, $c>0$ is the constant premium rate, $(N_t)_{t \geq 0}$ a Poisson process with rate $\lambda>0$ describing the number of claims until time $t$ and the sequence of i.i.d. claim sizes is denoted by $(Y_k)_{k \in \mathbb{N}}$ and is also independent of $(N_t)_{t \geq 0}$.
The process $(U_t)_{t \geq 0}$ describes the amount of surplus of an insurance portfolio indexed by time. 
Further, we assume that $\E[Y_1]=\mu>0$ and that the net profit condition $c > \lambda \mu$ is satisfied. We denote the distribution function of $Y_1$ by $F$ and set $\overline{F}(t):=1-F(t)$. 

In the classical setup, the time of ruin is defined by $T := \inf\{ t > 0 \colon U_t < 0 \}$ with $\inf \emptyset = \infty$.
Typically, one is first interested in the analysis of the classical ruin probabilities $\psi_\mathrm{cl}(u) := \P_u(T<\infty)$, as $u \to \infty$.
For an overview of the classical theory, we refer to \cite{Embrechts1997} and \cite{Dickson2016}.

Different ruin related quantities have attracted a lot of attention in the literature. 
Here, see for instance the well-cited work of Gerber and Shiu \cite{Gerber1998} and the vast number of papers that followed. 
Moreover, many extensions and modifications of the classical model have been established.
Again, in many situations, one is first interested in the corresponding questions from the classical setup.
In the recent literature, modified definitions of ruin are considered. 
For instance in \cite{Cai2007}, a model is studied where the insurance company can borrow money at a certain debit interest when $U_t$ is negative. 
Further, the concept of Parisian ruin has attracted a lot of attention in the literature. Here, the surplus process is allowed to stay negative for a continuous time interval of a fixed or random length, see \cite{Dassios2008}, \cite{Czarna2011}, \cite{Loeffen2013}, \cite{Landriault2014} and for the cumulative situation \cite{Guerin2017}. 
In omega models, the insurance company goes bankrupt at a random time at some surplus dependened bankruptcy rate when $U_t$ is negative, see \cite{Albrecher2011}, \cite{Gerber2012} and \cite{Albrecher2013}. 
This model is in turn linked to models where the insurance company can just go bankrupt at random observation times, see \cite{Albrecher2011a} and \cite{Albrecher2013a}.

The aim of the present note is to study the asymptotics of the ruin probability of a large class of models with modified notion of ruin. 
In contrast to the approaches in the recent literature, our technique does not require a specific model. We only need that, for $u \geq 0$,
\begin{align}
\begin{split}
\label{eq:model}
\psi(u) 
&=
\int_{-\infty}^0 \psi(y) \P_u( U_T \in \di y ,\, T < \infty ),
\end{split}
\end{align}
where $\psi(u)$ denotes the modified ruin probability, for initial capital $u$.
This assumption expresses that the mechanism that causes the ruin gets activated when the process hits the negative half-line. 
The general form of \eqref{eq:model} allows us to gather most models from the literature as well as many new models under one umbrella.

In order to define such a model and to verify \eqref{eq:model}, it is often natural to define a corresponding time of ruin $\tau$. 
Then, we set $\psi(u):=\P_u(\tau < \infty)$.
For example, in the situation of cumulative Parisian ruin (at level $r>0$), the modified ruin time is defined by $\tau := \inf\{ t>0 \colon \int_0^t \1{(-\infty,0)}(U_s) \di s > r \}$ and it follows immediately from the strong Markov property that \eqref{eq:model} is satisfied.
However, note that also every choice of a measurable function $\psi(u)$, for $u<0$, with values in $[0,1]$, defines such a model.
Further, note that the case $\psi(u)=1$, for $u<0$, coincides with the classical case. 

We proceed as follows. In Section \ref{sec:results}, we state and prove our main results. 
Then, we apply our results to a bunch of examples and give a short outlook in Section \ref{sec:examples}.

\section{Results}
\label{sec:results}

We investigate the two classical situations: Either the Cram\'er condition is fulfilled or the integrated claim sizes are subexponential.
Recall that the Cram\'er condition is satisfied for a constant $R>0$, if $ \lambda \E[ \exp(R Y_1) -1 ] = cR$. Then, it is well-known that $\psi_\mathrm{cl}(u) \sim k e^{-Ru}$ for some $k>0$, as $u \to \infty$, e.g.\ see Theorem 1.2.2 in \cite{Embrechts1997}.
Further, let $F_I$ be the distribution function defined by $F_I(t):=\frac{1}{\mu}\int_0^t \overline{F}(s) \di s$, for $t\geq 0$.
A distribution function $F$ is called subexponential if $\lim_{u \to \infty}\frac{\overline{F^{\ast 2}}(u)}{\overline{F}(u)} = 2$. In this case, we write $F \in \mathcal{S}$. If $F_I \in \mathcal{S}$, one has $\psi_\mathrm{cl}(u) \sim \frac{\lambda}{c-\lambda\mu}\int_u^\infty \overline{F}(z) \di z = \overline{F_I}(u)$, as $u \to \infty$, e.g.\ see Theorem 1.3.6 in \cite{Embrechts1997}.
We refer to this situation as heavy-tailed in the following.
For a discussion on subexponential distributions, see e.g.\ \cite{Embrechts1997}.

Our main theorem treats the relation of the asymptotics of modified ruin probabilities to the classical ruin probability. 
\begin{thm}
\label{thm:mainthm}
Let $\psi$ satisfy condition \eqref{eq:model}.
\begin{enumerate}
\item
Suppose the Cram\'er condition is fulfilled with parameter $R>0$. 
If $\psi$ is continuous or monotone on $(-\infty,0)$, 
then $\psi(u) \sim C\psi_\mathrm{cl}(u)$, as $u \to \infty$, where $C=\int_{-\infty}^0 \psi(y) \P_\infty(\di y)$ and $\P_\infty$ has distribution function $1-\frac{\lambda}{c-\lambda\mu} \int_0^\infty (e^{Rz}-1) \overline{F}(z+\cdot) \di z$.
\item 
If $F_I \in \mathcal{S}$ 
and $\lim_{u \to -\infty} \psi(u) = 1$, then $\psi(u) \sim \psi_\mathrm{cl}(u)$, as $u \to \infty$.
\end{enumerate}
\end{thm}

\begin{proof}
Due to \eqref{eq:model}, we have
\begin{align}
\label{eq:psiFormula}
\begin{split}
\psi(u) 
&=
\int_{-\infty}^0 \psi(y) \P_u( U_T \in \di y ,\, T < \infty )\\
&=
\psi_\mathrm{cl}(u) \int_{-\infty}^0 \psi(y) \P_u( U_T \in \di y \mid T < \infty ),
\end{split}
\end{align}
and the analysis of the asymptotic behavior of the modified ruin probabilities reduces to the analysis of the integral $\int_{-\infty}^0 \psi(y) \P_u(U_T \in \di y \mid T < \infty)$, as $u \to \infty$.
Our result is based on Theorem 2 in \cite{Schmidli1999}, which states that, if the limit
\begin{equation}
\label{eq:gamma}
\gamma(z)=\lim_{u \to \infty} \frac{\psi_\mathrm{cl}(u+z)}{\psi_\mathrm{cl}(u)}
\end{equation}
exists, then
\begin{align}
\label{eq:limitdistr}
\begin{split}
& \lim_{u \to \infty} \P_u( -U_T>x \mid T < \infty ) \\
&=
\frac{1}{c-\lambda \mu} \left( c\gamma(x) - \lambda \int_0^x \gamma(x-z) \overline{F}(z) \di z - \lambda \int_x^\infty \overline{F}(z) \di z \right).
\end{split}
\end{align} 

Let us first assume that the Cram\'er condition is fulfilled for $R>0$.
Since $\psi_\mathrm{cl}(u) \sim k e^{-Ru}$ for some $k>0$, as $u \to \infty$, the limit in \eqref{eq:gamma} exists with $\gamma(z)=e^{-Rz}$.
Now, following Example 2 in \cite{Schmidli1999}, one obtains that $\P_u( -U_T \in \cdot \mid T < \infty )$ converges in distribution to the probability measure $\P_\infty$ with distribution function
\begin{equation*}
x \mapsto 1 - \frac{\lambda}{c-\lambda\mu} \int_0^\infty (e^{Rz}-1) \overline{F}(z+x) \di z.
\end{equation*}
If $\psi$ is continuous on $(-\infty,0)$, the claim follows immediately from \eqref{eq:psiFormula}.
Since the limit distribution is continuous, the claim follows as well if $\psi$ is monotone on $(-\infty,0)$.

Now, if $F_I \in \mathcal{S}$, 
one has $\psi_\mathrm{cl}(u) \sim \frac{\lambda}{c-\lambda\mu}\int_u^\infty \overline{F}(z) \di z = \overline{F_I}(u)$, as $u \to \infty$.  
Since, by Lemma 1.3.5 in \cite{Embrechts1997}, $F_I$ is long-tailed, we have $\gamma(z)=1$ in \eqref{eq:gamma}.
Therefore, for all $x\geq 0$,
\begin{equation*}
\lim_{u \to \infty} \P_u( -U_T>x \mid T < \infty )=1.
\end{equation*}
Thus, for any $x \geq 0$,
\begin{align*}
\int_{-\infty}^0 \psi(y) \P_u( U_T \in \di y \mid T < \infty )
&\geq
\inf_{y < -x} \psi(y) \P_u( -U_T > x \mid T < \infty ) \\
&\to \inf_{y < -x} \psi(y),
\end{align*}
as $u \to \infty$. 
Since $x$ was arbitrary, we can let $x \to \infty$, and the claim follows.
\end{proof}

\begin{rem}
Theorem \ref{thm:mainthm} is particularly useful in the heavy-tailed case, since 
without computing $\psi(u)$ explicitly, for $u<0$, one obtains exact asymptotic results for the modified ruin probabilities, as long as one knows that $\lim_{u \to - \infty} \psi(u) = 1$. 
This condition is very natural, since in most situations, it should become impossible to survive with negative surplus $y$, when $y \to -\infty$.
Likewise, without computing $\psi(u)$ explicitly, for $u<0$, one obtains that the modified and classical ruin probabilities differ asymptotically by a constant $C$, if the Cram\'er condition is fulfilled and 
$\psi$ is continuous or monotone on $(-\infty,0)$.
Again, e.g.\ the monotonicity assumption is very natural for similar reasons as above, since in most situations, it should become harder to survive when the surplus becomes more negative. 
Of course, in contrast to the heavy-tailed case, more knowledge of $\psi(u)$, for $u<0$, is required in order to compute the constant $C$.
\end{rem}

\begin{rem}
The proof of Theorem \ref{thm:mainthm} hinges on the limit theorem for the probability measure $\P_u( U_T \in \cdot \mid T<\infty )$, which leads to asymptotic results. 
For more precise results, more information about these probability measures is required. 
For example, explicit results (in terms of $\psi(u)$ for $u<0$) can be obtained if the claim sizes are $\exp(\delta)$-distributed. In this case, closed form expressions for the classical ruin probability $\psi_\mathrm{cl}(u)$ are available and $\P_u( U_T \in \cdot \mid T < \infty )$ is again $\exp(\delta)$-distributed, e.g.\ see \cite{Asmussen2000}.
\end{rem}

In many modified ruin models from the literature, 
the process starts renewed after surviving an excursion in the negative half-line. 
More precisely, in such situations, one has $1-\psi(y) = p_y(1-\psi(0))$ with $p_y := \P_y(T_0 < \tau)$, for $y<0$, where $T_0 := \inf\{ t>0 \colon U_t=0 \}$. That means, if the process survives until it reaches zero after becoming negative, the process starts renewed and survives afterwards with probability $1-\psi(0)$. 
In the following proposition, we will give expressions for the modified ruin probability $\psi(u)$ in terms of $p_y$.

\begin{prop}
\label{prop:psiRenew}
If $1-\psi(y) = p_y(1-\psi(0))$, for $y<0$, one has
\begin{equation*}
q_0 := 1 - \psi(0) = \frac{1-\psi_\mathrm{cl}(0)}{1-p_0}
\end{equation*}
with $p_0 = \P_0(T_0 < \tau , T<\infty)$ and
\begin{equation}
\label{eq:psiRenew}
\psi(u) 
= \psi_\mathrm{cl}(u) \left( 1 - q_0 \int_{-\infty}^0 p_y \P_u( U_T \in \di y \mid T<\infty) \right)
.
\end{equation}
\end{prop}

\begin{proof}
By \eqref{eq:model} and the assumption, we obtain 
\begin{align}
\label{eq:proofLemma}
1-\psi(u)
&=
1-\int_{-\infty}^0 \psi(y) \P_u( U_T \in \di y ,\, T < \infty ) \nonumber \\
&=
1-\psi_\mathrm{cl}(u) + \int_{-\infty}^0 (1-\psi(y)) \P_u( U_T \in \di y ,\, T < \infty ) \nonumber \\
&=
1-\psi_\mathrm{cl}(u) + (1-\psi(0)) \int_{-\infty}^0 p_y \P_u( U_T \in \di y ,\, T < \infty ).
\end{align}
For $u=0$, it follows that
\begin{equation*}
1-\psi(0) = (1-\psi_\mathrm{cl}(0)) + (1-\psi(0)) \int_{-\infty}^0 p_y \P_0( U_T \in \di y ,\, T < \infty ),
\end{equation*}
and thus,
\begin{equation*}
1-\psi(0) = \frac{1-\psi_\mathrm{cl}(0)}{1-p_0}
\end{equation*}
with $p_0 = \P_0(T_0 < \tau ,\, T<\infty)=\int_{-\infty}^0 p_y \P_0(U_T \in \di y ,\, T<\infty)$.
Now, \eqref{eq:psiRenew} follows from \eqref{eq:proofLemma}.
\end{proof}

\begin{rem}
Since $\P_0(U_T \in \cdot \mid T<\infty)$ has the distribution function $F_I$, e.g.\ see Proposition 8.3.2 in \cite{Embrechts1997}, 
an explicit expression for $p_0$ in terms of $p_y$, for $y<0$, is available.
Further, it is known that $\psi_\mathrm{cl}(0)=\frac{\mu \lambda}{c}$, e.g.\ see p.~31 in \cite{Embrechts1997}.
This together with Proposition \ref{prop:psiRenew} gives an explicit expression for $\psi(0)$ in terms of $p_y$, for $y<0$.
\end{rem}

\begin{rem}
\label{rem:perspective}
The formulation of the modified ruin probability in \eqref{eq:psiRenew} in terms of $p_y$ leads to a new perspective: We can think of $p_y$, for $y<0$, as the probability of finding an investor when the surplus drops below zero that pays until recovery. 
This perspective is also a natural starting point to build new models 
in the sense that any measurable function $p_y$ on $(-\infty,0)$ defines a model with modified definition of ruin and the preceding interpretation.
\end{rem}

\begin{rem}
Proposition \ref{prop:psiRenew} gives us an exact expression for the modified ruin probability $\psi(u)$ in terms of $p_y$, for $y<0$, and the probability measure $\P_u(U_T \in \cdot \mid T<\infty)$.  
Combining this result with Theorem \ref{thm:mainthm}, we obtain that, 
if the Cram\'er condition is fulfilled and if
$p_y$ is continuous or monotone, one has $\psi(u) \sim C\psi_\mathrm{cl}(u)$, as $u \to \infty$, with $C=1 - q_0 \int_{-\infty}^0 p_y \P_\infty( \di y )$.
The condition $\lim_{u\to-\infty}\psi(u)=1$ in the second part of Theorem \ref{thm:mainthm}  translates now into the condition $\lim_{y \to - \infty }p_y = 0$. In this case $\psi(u) \sim \psi_\mathrm{cl}(u)$, as $u\to\infty$.
Again, we emphasize at this point that the above assumptions are quite natural. For example, it is natural to assume that $p_y$ is monotone, since it should be harder to find an investor when the surplus becomes more negative. Similarly, it should become impossible to find an investor with negative surplus $y$, as $y \to -\infty$.
\end{rem}

\section{Examples and outlook}
\label{sec:examples}

We will give examples and show that our results can be applied to many established models from the literature.
\begin{ex}
We choose $p_y=p \in [0,1]$, for $y<0$. 
Then, $p_0 = p \frac{\mu \lambda}{c}$, and thus, by \eqref{eq:psiRenew},
\[ \psi(u) = \psi_\mathrm{cl}(u) \left( 1 - p \frac{1-\frac{\mu \lambda}{c}}{1-p\frac{\mu \lambda}{c}} \right). \]
This example corresponds to the situation where the probability of finding an investor does not depend on $U_T$.
\end{ex}

\begin{ex}
If 
the claim sizes are $\exp(\delta)$-distributed, one obtains straightforwardly from \eqref{eq:psiRenew} for arbitrary $p_y$ that
\[
\psi(u) 
= \psi_\mathrm{cl}(u) \frac{1- {c \delta p_0}/{ \lambda} }{1-p_0}, \quad \text{with } p_0 = \frac{ \lambda}{c \delta}  \int_{-\infty}^0 p_y \delta e^{\delta y} \di y .
\] 
\end{ex}
Next, we will see that our results can be applied to most models with a modified definition of ruin from the literature. 
\begin{ex}
\label{ex:Parisian}
First, let us recall the definitions of Parisian ruin and cumulative Parisian ruin, respectively.
Let $g_t := \sup\{ s \leq t \colon U_s \geq 0 \}$. Then, the time of Parisian ruin (at level $r>0$) is defined as \[ \tau := \inf\{ t > 0 \colon t-g_t > r \}. \]
The time of cumulative Parisian ruin (at level $r>0$) is defined as \[\tau := \inf\left\{ t>0 \colon \int_0^t \1{(-\infty,0)}(U_s) \di s > r \right\}.\] 
If the constant $r$ is replaced by an independent exponentially distributed random variable, one obtains the definition of exponential (cumulative) Parisian ruin. 
In these cases, it is straightforward to verify the assumption on $\psi(u)$, for $u<0$, in Theorem \ref{thm:mainthm}.
Recall that explicit expressions of the cumulative Parisian ruin probabilities are given in \cite{Guerin2017} for exponentially distributed claim sizes. Our results extend these results to asymptotic results for more general claim size distributions if the Cram\'er condition is satisfied or if $F_I \in \mathcal{S}$.
\end{ex}

\begin{ex}
\label{ex:Omega}
Our results can be applied to omega models. Let $\omega \colon \R \to \R$ be a monotonically non-increasing function with $\omega(y)=0$, for $y \geq 0$, and $\omega(y)>0$, for $y<0$. Let $\mathrm{e}_1$ be an $\exp(1)$-distributed random variable independent of $(U_t)_{t\geq 0}$. Then, in an omega model, the time of ruin is defined as \[ \tau := \inf\left\{ t>0 \colon \int_0^t \omega(U_s) \di s > \mathrm{e}_1 \right\}. \] Thus, $\omega$ describes the bankruptcy rate in the model.
It is straightforward to verify that $\psi(u)>0$ and $\lim_{u \to -\infty} \psi(u) = 1$ and we can apply Theorem~\ref{thm:mainthm}.  
Thus particularly, we extend the results in \cite{Albrecher2013} -- where the authors restricted themselves to exponentially distributed claim sizes -- to asymptotic results if  the Cram\'er condition is fulfilled or if $F_I \in \mathcal{S}$. 
\end{ex}

\begin{rem}
If the bankruptcy rate in Example \ref{ex:Omega} is constant for $y<0$, the process can only stay exponential times in the negative half-line. Thus, we have the same situation as in the exponential (cumulative) Parisian ruin model in Example \ref{ex:Parisian}. Due to the memorylessness of the exponential distribution, this situation coincides further with a model where 
the insurance company can only go bankrupt after independent exponential times, see \cite{Albrecher2011a}, \cite{Albrecher2013a}. 
For this connection and more motivation for omega models, see \cite{Albrecher2013}. Further, if $\omega$ is constant for $y<0$, it is not hard to show that $p_y = e^{\gamma y}$, for $y<0$ and some $\gamma > 0$ depending on the bankruptcy rate and $F$. Hence, we have four different pictures in this case: Omega model with constant bankruptcy rate, exponential (cumulative) Parisian ruin, random observation times (with exponential times between observations) and a model where the probability of finding an investor decays exponentially. 
\end{rem}

\begin{ex}
In the model considered in \cite{Cai2007}, the insurance company can borrow money at a fixed debit interest rate when $U_t$ is negative. 
Clearly, if the surplus is below a certain negative level, the due interest exceeds the income of the insurance company and a recovery is impossible. 
Hence, in terms of our model, $\psi(y)$ and $p_y$ take the value $1$ and $0$, respectively, below this negative level. 
Thus,  Theorem \ref{thm:mainthm} can be applied.  
Moreover, we improve Theorem 4.1 in \cite{Cai2007}, since we can drop some of the technical assumptions there.
\end{ex}

Finally, let us give a short outlook.
In this paper, the Cram\'er-Lundberg model was considered to demonstrate our technique. We have proved that under the natural assumptions in Theorem \ref{thm:mainthm}, classical and modified ruin probabilities differ asymptotically by a constant factor if the Cram\'er condition is satisfied and are asymptotically equivalent if $F_I \in \mathcal{S}$.

There are many ways our results can be generalized.
Generally, as soon as limit theorems similar to \eqref{eq:limitdistr} are available for a process, corresponding results can be obtained. 
For instance, one can involve further quantities that affect the mechanism that causes the ruin. 
For example, the quantity $U_{T-}$ can be easily involved, e.g.\ see \cite{Schmidli1999}. However, there are, to the knowledge of the authors, no modified ruin definitions in the current literature using this quantity.  

Another direction is to consider different types of processes. 
It seems natural to consider spectrally negative L\'evy processes and processes that are perturbed by a Brownian motion. In the latter case, the process does not necessarily enter the negative half-line with a jump, and thus, this event would require additional techniques.

\end{document}